\documentclass[12pt]{amsart}
\usepackage[a4paper]{geometry}

\usepackage{tabularx}

\usepackage[all]{xy}
\usepackage{comment}
\usepackage{graphicx} 
\usepackage[T1]{fontenc}
\usepackage{hyperref}
\usepackage{array}   
    \newcolumntype{L}{>{$}l<{$}} 
\usepackage{amssymb}
\usepackage{amsthm}
\numberwithin{equation}{section}

\usepackage{mathrsfs}
\usepackage{float}
\usepackage{enumitem}

\newtheorem{theorem}{Theorem}[section]
\newtheorem{proposition}[theorem]{Proposition}
\usepackage[dvipsnames]{xcolor}

\newtheorem{corollary}[theorem]{Corollary}
\newtheorem{lemma}[theorem]{Lemma}
\theoremstyle{definition}

\theoremstyle{remark}
\newtheorem{remark}{Remark}
\usepackage{MnSymbol}

\usepackage{dsfont}
\newcommand\bigsubset[1][1.19]{%
   \mathrel{\vcenter{\hbox{\scalebox{#1}{$\subset$}}}}}

\newcommand{\FF}{\mathbb{F}}
\newcommand{\ZZ}{\mathbb{Z}}

\newcommand{\NN}{\mathbb{N}}
\newcommand{\lcm}{\mathrm{lcm}}
\newcommand{\ps}[1]{\left( #1 \right)}

\title
{Product Sets of Arithmetic Progressions  in Function Fields}
\author{Lior Bary-Soroker}
\address{School of Mathematical Sciences, Tel Aviv University, Tel Aviv 69978, Israel}
\email{barylior@tauex.tau.ac.il}
\author{Noam Goldgraber}
\address{School of Mathematical Sciences, Tel Aviv University, Tel Aviv 69978, Israel}
\email{noam3goldgraber@gmail.com}

\date{\today}

\begin{document}

\begin{abstract}
    We study product sets of finite arithmetic progressions of polynomials over a finite field. We prove a lower bound for the size of the product set, uniform in a wide range of parameters. We apply our results to resolve the function field variants of Erd\H{o}s' multiplication table problem.
\end{abstract}

\maketitle

\section{Introduction}
    \subsection{Main Result}
    Xu and Zhou \cite{product_sets_of_APs} establish a general lower bound for the size of the product set of arithmetic progressions:   Let $\mathcal{A}_1$, $\mathcal{A}_2 \subset \ZZ$ be two finite arithmetic progressions with lengths $2 \leq |\mathcal{A}_1| \leq |\mathcal{A}_2|$. Then
    \begin{equation}\label{XuZhou}
               |\mathcal{A}_1 \cdot \mathcal{A}_2| \geq |\mathcal{A}_1||\mathcal{A}_2|(\log |\mathcal{A}_1|)^{-\delta/2 - o(1)}(\log |\mathcal{A}_2|)^{-\delta/2 - o(1)},
    \end{equation}
        as $ |\mathcal{A}| \to \infty$, where 
    \begin{equation}
        \delta = 1-\frac{1+\log \log 2}{\log 2} = 0.08607\ldots  
    \end{equation}
    This implies the Elekes-Ruzsa conjecture, and is
    motivated by Erd\H{o}s' Multiplication Table Problem (MTP) and Ford's \cite{Ford_MT} resolution of it. We shall expand on the motivation below.
    
    We work in the function field setting, in which the ring of integers is modeled by the ring of polynomials $\FF_q[T]$ over a finite field $\FF_q$ with $q$ elements. Our objective is to improve the lower bound of Xu and Zhou in the function field setting (even though the Xu-Zhou result was not transferred to function fields, to the best of our knowledge). To obtain this goal, we restrict the uniformity of the arithmetic progressions. 
    
    To state our results, we introduce some notation. We denote by $\mathcal{M}_b\subseteq \FF_q[T]$ the set of monic polynomials of degree $b$. For polynomials $F,G \in \FF_q[T]$, we use the standard notation $(F,G) := \gcd (F,G)$, $[F,G] = \lcm(F,G)$. Given a polynomial $M\in \mathcal{M}_n$, we let $|M|=q^{\deg M}$, $\Phi(M) = \#\{ A: \deg A<\deg M, (A,M)=1\}$, and $P_j(M)$ the number of distinct monic irreducible divisors of $M$ of degree $j$.
    For two functions $f,g$, we write $f\gg g$  if there exists a constant $C>0$ such that $|f(x)|\geq C|g(x)|$ for all $x$. Moreover, we write $f(x) \asymp g(x)$ if $f(x) \ll g(x)$ and $f(x) \gg g(x)$. If the constant $C$ depends on some parameters, say $a,b,\ldots$, we indicate this by writing $f\gg_{a,b,\ldots} g$ and $f\asymp_{a,b,\ldots} g$ and we mean that it depends \emph{ only} on these parameters.

    In the function field setting, an arithmetic progression of modulus $M\in \FF_q[T]$ is defined as $\{ G \in \FF_q[T]: G\equiv A \mod M\}$. The logarithmic function is expressed as $\log_{q}$, which is analogous to the natural logarithm in \eqref{XuZhou}. 

    \begin{theorem}\label{thm: divisors on APs}
        Let $0 < \epsilon$, $0 < C$ and $1\leq b=b(n)\leq n/2$. Let $M_1,M_2, A_1,A_2 \in \FF_q[T]$ and $M=[M_1,M_2]$ such that $(A_i,M)=1$.
        Let $\Omega_i = \{ G\in \mathcal{M}_{b_i}: G \equiv A_i\mod M_i\}$ where $b_1=b$ and $b_2=n-b$. 
        Assume that 
        \begin{enumerate}[label={\arabic*.}]
            \item $\deg M_1 , \deg M_2\leq (1/2-\epsilon) b$.
            \item For all $j\in\NN$, we have $P_j(M)\leq C \frac{q^j}   {j^{1+\epsilon}}$. \label{eq:consition_Pj}
        \end{enumerate}
        Then, 
        \begin{equation}\label{mainresult}
            |\Omega_1\cdot \Omega_2| \gg_{C,\epsilon} \frac{q^n}{\Phi(M_1)  \Phi(M_2)}\cdot \frac{1}{b^{\delta}(\log b)^{\frac32}},
        \end{equation}
        as $q^n\to \infty$, where the implied constant depends only on $C$ and $\epsilon$ (and not on $q,M_1,M_2,n,b$).
    \end{theorem}    

    In the setting of the theorem, an analog to \eqref{XuZhou} would give the lower bound:
    \[
        |\Omega_1\cdot \Omega_2|\geq \frac{q^n}{|M_1||M_2|} \cdot\frac{1}{b^{\delta/2+o(1)}(n-b)^{\delta/2+o(1)}}.
    \]
    Under condition~\ref{eq:consition_Pj}, $|M_i|\ll \Phi(M_i)\leq |M_i|$, so we can replace $|M_i|$ by $\Phi(M_i)$, and there is no difference there.
    The two main improvements in \eqref{mainresult} are the following. First, we make the $o(1)$ in the exponent explicit. Instead of $b^{o(1)}$, which is explicitly stated in \cite{product_sets_of_APs} as the integer analog to $(\log b)^{7+o(1)}$ -- we have $(\log b)^{3/2}$. Second, we have the term $b^{\delta}$ in the denominator that is larger than the term $b^{\delta/2}(n-b)^{\delta/2}$ appearing in \cite{product_sets_of_APs}. This improvement is significant whenever $b=o(n)$.
    We emphasize that Theorem~\ref{thm: divisors on APs} is uniform in $q$, that is, one may vary the finite field, a feature that does not arise in $\ZZ$.
    \begin{remark}
        Condition \ref{eq:consition_Pj}\ of Theorem~\ref{thm: divisors on APs} is valid for $j \geq \frac{\log(n)}{\log(q)}$, and by increasing the value of $C$, it can be made to hold for all $j \leq j_0$ for some fixed $j_0$. This condition is intended to prevent polynomials $M$ that are composed of many small primes.
    \end{remark}
    \subsection{Motivation and Applications}
    \subsubsection{The classical multiplication table problem}
    Let $A(x)$ be the number of positive integers $n \leq x$ that can be written as $n=m_1m_2$ with each $m_1,m_2 \leq \sqrt{x}$. In other words, $A(x)$ is the number of distinct positive integers shown in the multiplication table $\sqrt{x} \times \sqrt{x}$. The multiplication table problem (MTP) asks for the asymptotic behavior of $A(x)$. It goes back at least to Erd\H{o}s \cite{Erdos_MT}. Subsequently, Ford \cite{Ford_MT} solved it completely. For an extensive review of the problem, see the introduction of \cite{Ford_MT}.  
    Ford proved that 
    \begin{equation}\label{eq:Ford}
        A(x) \asymp \frac{x}{(\log x)^\delta (\log \log x)^{3/2}}.
    \end{equation}

    To establish \eqref{eq:Ford}, Ford considered a more subtle quantity. For $y\leq z\leq x$ let
    \[
        H(x,y,z) = \{ n\leq x : \exists d\in (y,z] \mbox{ such that } d\mid n\}.
    \]
    Following works of Besicovitch \cite{besicovitch1935density}, Erd\H{o}s \cite{Erdos_MT}, and Tenenbaum  \cite{tenenbaum1980lois,tenenbaum1975repartition}, Ford gave a uniform asymptotic formula \cite[Theorem~1]{Ford_MT} which in particular implies that 
    \begin{equation}\label{eq:Hy2y}
        |H(x,y,2y)| \asymp \frac{x}{(\log x)^\delta (\log \log x)^{3/2}}.
    \end{equation}
    Note that \eqref{eq:Ford} follows from \eqref{eq:Hy2y}, using the inequalities
    \[
        \left|H\left(\frac{x}{4}, \frac{\sqrt{x}}{4}, \frac{\sqrt{x}}{2} \right) \right| \leq |A(x)| \leq \sum_{k \geq 0} \left| H \left(\frac{x}{2^k}, \frac{\sqrt{x}}{2^{k+1}}, \frac{\sqrt{x}}{2^k} \right) \right|.
    \]

    Variants of the MTP were studied extensively in the literature: For example, in \cite{Koukoulopoulos_generalized_MT,Koukoupoulos_localized_factorization}, Koukoulopoulos 
    studies a higher dimensional variant of the MTP, and in \cite{Koukoulopoulos_Divisors_of_shifted_primes} he restricts the set of integers to shifts of primes (e.g., $n\in \{p+2: p \mbox{ prime}\}$). Mehdizadeh \cite{smooth_numbers_MT} studies the  number of $y$-smooth numbers in the MTP. The problem was also considered in other settings: 
    Eberhard, Ford and Green \cite{permutations_MT} studies an analogue of the MTP for permutations, Narayanan,  Sahasrabudhe, and Tomon \cite{bipartite_grphs_MT} for bipartite graphs, and Meisner \cite{Meisner_MT} for polynomials over finite fields. 

    \subsubsection{Function field variants of the MTP}
    Recall that  $\mathbb{F}_q$ is the finite field with $q$  elements, $\FF_q[T]$ the ring of polynomials with coefficients in $\FF_q$, and $\mathcal{M}_n \subset \mathbb{F}_q[T]$ the subset of monic polynomials of degree $n$. Moreover, we denote $\mathcal{M}=\bigcup_{n\in \NN} \mathcal{M}_n$. There is a classical analogy between $\FF_q[T]$ and $\ZZ$. In particular, the analogue of $|a| = | \mathbb{Z}/a\mathbb{Z}|$, for $a\in \mathbb{Z}$ is $|f| = |\FF_q[T]/f\FF_q[T]| = q^{\deg f}$, for $0\neq f \in \FF_q[T]$. In addition, $\mathcal{M}_n$ may be considered as the analogue of the interval $(x,2x]$, where $q^n = |\mathcal{M}_n|$ is analogous to $x\sim |(x,2x]\cap \ZZ|$. 

    This led Meisner \cite{Meisner_MT} to define 
    \begin{align}
        \label{eq:M2n}
        M(2n) &:= \{ F \in \mathcal{M}_{2n} : \exists G\in \mathcal{M}_n \mbox{ such that } G\mid F\}, \\
        \label{eq:Hnb}
        H(n,b) &:= \{ F \in \mathcal{M}_n : \exists G\in \mathcal{M}_b \mbox{ such that } G\mid F\}.
    \end{align}
    Then, Meisner considers $|M(2n)|$ as an analogue of $A(x)$, and $|H(n,b)|$ as an analogue of $|H(x,y,2y)|$. He proves\footnote{
        The proof in  \cite{Meisner_MT}  contains a gap: It uses a uniform-in-$q$ estimate of the number of rough polynomials. However, \emph{loc.\ cit.} omits a proof, and only sketches an argument,  which the authors could not complete. We fill in the gap in Proposition~\ref{prop:Uniform Rough Polynomials} using a different argument based on Selberg's sieve.} that, for $b\leq n/2$,
    \begin{equation}\label{Meisner_Hnb}
        |H(n,b)| \asymp \frac{q^{n}}{b^\delta(1+ \log b)^{3/2}}, \qquad q^n \to \infty.
    \end{equation}
    See \cite[Theorem 1.2]{Meisner_MT}.

    Since $M(2n)=H(2n,n)$, the MTP follows directly from \eqref{Meisner_Hnb}:
    \[
            |M(2n)| \asymp \frac{q^{2n}}{n^\delta(1+ \log n)^{3/2}}, \qquad q^n\to \infty.
    \]

    One natural variant of the MTP is to consider only these elements in the table that fall into a given arithmetic progression.
    For this purpose, for polynomials $M \in \mathbb{F}_q[T]$ and $A \in \mathbb{F}_q[T] / M\mathbb{F}_q[T]$, let us define the set
    \begin{align*}
        H(n,b;A,M) := \left\{F\in \mathcal{M}_n : F \mbox{ has a divisor of degree b}, F \equiv A \mod\ M\right\}.
    \end{align*}
    We expect that $H(n,b;A,M)$ distributes uniformly for residues $A\in (\FF_q[T]/M\FF_q[T])^*$. We show that the expected lower bound holds, at least up to a constant, under some mild restrictions:
    \begin{theorem}\label{thm:MTP on AP}
        Let $0 < \epsilon$, $0 < C$ and $1\leq b=b(n)\leq n/2$. Then
        \[
            |H(n,b;A,M)| \gg_{C,\epsilon} \frac{1}{\Phi(M)}\cdot \frac{q^n}{b^{\delta}(1+\log b)^{\frac32}}, \qquad q^n\to \infty,
        \]
        uniformly over all $M,A\in \FF_q[T]$ and $(A,M)=1$ such that:
        \begin{align}
             &\deg M \leq (1/2 - \epsilon) b , \label{cond.1}\\
              &P_j(M)\leq C \frac{q^j}{j^{1+\epsilon}}, \quad \mbox{for all $j\in\NN$ }. \label{cond.2}
        \end{align}
    \end{theorem}
    The proof does not follow directly from Theorem~\ref{thm: divisors on APs}, but rather from the methods developed to prove it. 
    Another natural variant is to consider the divisors to be taken from a given arithmetic progression. Define
    \[
         H'(n,b;A,M) := \left\{F\in \mathcal{M}_n : \exists G\in \mathcal{M}_b, G|F \mbox{ and }G \equiv A \mod\ M\right\}.
    \]
    We obtain lower bounds under similar assumptions.
    \begin{theorem}\label{thm:divisor on AP}
        Let $0 < \epsilon$, $0 < C$ and $1\leq b=b(n)\leq n/2$. Then 
        \[
            |H'(n,b;A,M)| \gg_{C,\epsilon} \frac{1}{\Phi(M)}\cdot \frac{q^n}{b^{\delta}(1+\log b)^{\frac32}}, \qquad q^n\to \infty,
        \]
        uniformly over all $M,A\in \FF_q[T]$ and $(A,M)=1$ satisfying \eqref{cond.1}, \eqref{cond.2}.
    \end{theorem}

    Theorem \ref{thm:divisor on AP} follows from Theorem~\ref{thm: divisors on APs}, see Section~\ref{sec:last} for a formal derivation.

     These theorems may be applied to resolve  the MTP restricted to arithmetic progressions. We derive the following two generalizations of multiplication tables: The first concerns 
    \[
        M(2n;A,M) := H(2n,n;A,M).
    \]
    That is, we keep only the entries $\equiv A\mod M$ in the  original table. The second concerns 
    \[
        M'(2n;A,M) := H'(2n,n; A,M),
    \]
    i.e. we take one of the sides consisting of elements $\equiv A\mod M$. Theorems~\ref{thm:MTP on AP} and \ref{thm:divisor on AP} immediately imply the following result.
    
    \begin{corollary}
        Let $0 < \epsilon$, $0 < C$ and $1\leq b=b(n)\leq n/2$. Then 
        \[
            |M(2n;A,M)| \gg_{C,\epsilon} \frac{1}{\Phi(M)}\cdot \frac{q^{2n}}{n^{\delta}(1+\log n)^{\frac32}}, \qquad q^n\to \infty,
        \]
        \[
            |M'(2n;A,M)| \gg_{C,\epsilon} \frac{1}{\Phi(M)}\cdot \frac{q^{2n}}{n^{\delta}(1+\log n)^{\frac32}}, \qquad q^n\to \infty,
        \]
        uniformly over all $M,A\in \FF_q[T]$ and $(A,M)=1$ satisfying \eqref{cond.1}, \eqref{cond.2}.
    \end{corollary}

    The last variant we consider is where both sides consist of arithmetic progressions. Let
    \[
    \begin{split}
        H(n,b;A_1,A_2,M_1,M_2) &:=\{ G_1G_2: G_i\equiv A_i\mod M_i,\ G_1\in \mathcal{M}_b, G_2 \in \mathcal{M}_{n-b}\}, \\
        M(2n;A_1,A_2,M_1,M_2) &:= H(2n,n;A_1,A_2,M_1,M_2).
    \end{split}
    \]

    \begin{figure}[H]
    \begin{tabular}{c|cc}
      $\times$  & $T^3+T+1$    & $T^3+T^2+T+1$      \\ \hline
      $T^3+1$  &  $T^6 + T^4 + T + 1$ &  $T^6 + T^5 + T^4 + T^2 + T + 1$ \\
      $T^3+T+1$  & $T^6 + T^2 + 1$ &  {\color{RoyalBlue}$T^6 + T^5  + T^3  + 1$}  \\
      $T^3+T^2+1$  & $T^6 + T^5 + T^4 + T^3 + T^2 + T + 1$ &  $T^6 + T^3 + T + 1$   \\
      $T^3+T^2+T+1$ & {\color{RoyalBlue} $T^6 + T^5 + T^3 + 1$} & $T^6 + T^4  + T^2  + 1$ 
    
    \end{tabular}
    \caption{Repetitions in the Multiplication Table for $M(6;1,T+1,T,T^2)$ over $\FF_2[T]$.}
    \end{figure}
    Then Theorem~\ref{thm: divisors on APs} implies:
    \begin{theorem}
        Let $0 < \epsilon$, $0 < C$ and $1\leq b=b(n)\leq n/2$. Let $M_1,M_2 \in \FF_q[T]$, and denote $M=[M_1,M_2]$. Then
        \[
            |M(2n;A_1,A_2,M_1,M_2)| \gg_{C,\epsilon} \frac{1}{\Phi(M_1) \cdot \Phi(M_2)}\cdot \frac{q^{2n}}{n^{\delta}(1+\log n)^{\frac32}}, \qquad q^n\to \infty,
        \]
        uniformly over all $M_i,A_i\in \FF_q[T]$ and $(A_i,M)=1,\ i=1,2$ and such that both $M_1, M_2$ satisfy \eqref{cond.1}, \eqref{cond.2}.
    \end{theorem}
    We apply \cite{Meisner_MT} to get sharp upper bounds only when the respective moduli are bounded and only as $n \to \infty$ and $q$ is fixed. Let us state the result for this case.
    \begin{corollary}\label{thm:M is bounded}
        Let $1\leq b=b(n)\leq n/2$, and fix $M_i,A_i\in \FF_q[T],\ i=1,2$. Assume that both $M_1,M_2$ satisfy \eqref{cond.1}. 
         Let $M:=[M_1,M_2]$ and assume $(A_i,M)=1$. Then,
        \begin{align*}
            |H(n,b;A_1,M_1)| \asymp |H'(n,b;A_1,M_1)| \asymp 
            |H(n,b;A_1,A_2,M_1,M_2)| \asymp \frac{q^{n}}{b^{\delta}(1+\log b)^{\frac32}}, \\
            |M(2n;A_1,M_1)| \asymp |M'(2n;A_1,M_1)| \asymp 
            |M(2n;A_1,A_2,M_1,M_2)| \asymp \frac{q^{2n}}{n^{\delta}(1+\log n)^{\frac32}},
        \end{align*}
        as $n \to \infty$.
    \end{corollary}

\subsection*{Outline of The Paper}
    Section \ref{section:rough polynomials} proves uniform estimates for the number of rough polynomials. Sections~\ref{sec:arithmetic_lb} and \ref{sec:avg_deg_div} contain the core of the proof of Theorems~\ref{thm: divisors on APs} and \ref{thm:MTP on AP}. In section \ref{sec:last} we complete the proof of Theorems \ref{thm: divisors on APs} and \ref{thm:MTP on AP}, and we deduce Theorems \ref{thm:divisor on AP} and \ref{thm:M is bounded}.
\section*{Acknowledgments}
    The authors are grateful to Ofir Gorodetsky for conversations about rough polynomials and their connection to permutations and for guidance on the relevant literature, and to Dimitris Koukoulopoulos for proposing the use of sieve theory in the proof of Theorem~\ref{prop:Uniform Rough Polynomials} and directing them to the paper by Xu and Zhou.

    The authors were partially supported by a grant of the Israel Science Foundation no.\ 702/19.

\section{On The Number of Rough Polynomials}\label{section:rough polynomials}
The goal of this section is to state results on $b$-rough polynomials that are needed for the main theorems, and to provide elementary and self-contained proofs. We also explain how these results may be derived from the literature.

Let $\mathbb{F}_q$ be a finite field with $q$ elements. Throughout this paper, we use the terminology \textbf{prime} to indicate a monic irreducible polynomial in $\mathbb{F}_q[T]$ -- we denote primes by the letter $P$. For $F\in \FF_q[T]$, we denote by $P^-(F)$ the smallest degree of a prime $P|F$.
For $b\in \mathbb{R}$, we say that a polynomial $F \in \FF_q[T]$ is $b$-rough if $P^-(F) > b$. For $n\in \NN$, we denote the number of $b$-rough monic polynomials of degree $n$ by
\[
    \Psi(n,b) := \# \{F\in \mathcal{M}_n: P^- (F) > b\}.
\]
For $A,M \in \FF_q[T]$, we define
\[
    \Psi(n,b; A,M) := \# \{F\in \mathcal{M}_n: P^- (F) > b,\ F \equiv A \mod{M}\}.
\]

 In this section, we determine the order of magnitude of $\Psi(n,b)$ and $\Psi(n,b;A,M)$ in a wide range of uniformity. We begin with $\Psi(n,b)$, and then show that $b$-rough polynomials are equidistributed -- at least up to a constant -- in arithmetic progressions with $(A,M)=1$, and hence deducing the order of magnitude of $\Psi(n,b;A,M)$.
 We introduce two more pieces of notation.

 For $n\in \NN$ and co-prime polynomials $A,M \in \FF_q[T]$, let
 \begin{align}\label{df:pi}
     \pi_q(n) &= \# \{P\in \mathcal{M}_n: P \textnormal{ is prime} \}, \\
     \pi_q(n;A,M) &= \# \{ P \in \mathcal{M}_n: P \textnormal{ is prime and } P \equiv A \mod M \}.
 \end{align}

We shall repeatedly use the following well-known strong forms of the prime polynomial theorems:

\begin{theorem} \label{th:PPTs}
    Let $M\in \mathbb{F}_q[T]$, $A \in \left(\mathbb{F}_q[T] / M \mathbb{F}_q[T]\right)^{\times}$ and let $n$  be a positive integer. Then,  
    \begin{align}
        \label{PPT}\frac{q^n}{n} - \frac{2q^{n/2}}{n} &\leq \pi_q(n) \leq \frac{q^n}{n}, \\
        \label{PPTAP}
        \pi_q(n;A,M)  &= \frac{1}{\Phi(M)}\frac{q^n}{n} + O\ps{\frac{q^{\frac{n}{2}}}{n}\deg M}, \qquad q^n \to \infty
    \end{align}
\end{theorem}

The first assertion follows from the exact formula $\pi_q(n) = \frac{1}{n} \sum_{d\mid n} \mu(n/d)q^d$, see \cite[Lemma~4]{pollack2013irreducible} and the second from the Riemann hypothesis for function fields, see \cite[Theorem 4.8]{NumberTheory_FF}.
\\

Warlimont \cite[Eq.\ 4]{Worlimont} estimated the number of $b$-rough polynomials for a fixed finite field. Using the comparison between polynomials and permutations \cite[Theorem 5.8]{ABT}, for large $q$, one may estimate $b$-rough polynomials by $b$-rough permutations, which is well known, see e.g.\ \cite{Ford_permutations}.

We give here an elementary direct proof, using Selberg's sieve.

\begin{theorem}\label{prop:Uniform Rough Polynomials}
    Let $1\leq b=b(n) < n$. Then,
    \[
    \Psi(n,b) \asymp \frac{q^n}{b}, \qquad q^n \to \infty.
    \]    

\end{theorem}

First, if $b > n/2$, then a polynomial is $b$-rough if and only if it is a prime, so $\Psi(n,b) = \pi_q(n)$ and we are done by \eqref{PPT}.
From now on assume $b\leq n/2$. 

The upper bound is given in a more general setting in \cite[Lemma~3.5]{GorodetskyKovaleva} using an elementary approach based on generating functions.

\begin{lemma}\label{lm:Sieve upper bound}
    Let $1\leq z=z(n) \leq n/2$ and q a prime power. Then,
    \[
    \Psi(n,z) \leq  \frac{q^n}{z(1-\frac{1}{q})}.
    \]    
\end{lemma}

\begin{proof}
    We introduce the sieve notation in our setting.
    \begin{enumerate}[label=\arabic*.]
        \item Let $\mathcal{A} = \mathcal{M}_n \subset \FF_q[T]$  and denote by $X := |\mathcal{A}| = q^n$. 
        \item Let $\mathcal{P} \subset \FF_q[T]$ be the set of all prime polynomials. 
        \item For $z>0$, define 
        \[
        \mathbf{P}(z) := \prod_{\substack{P\in \mathcal{P} \\ \deg P \leq z}} P.
        \]
        \item For $D \in \FF_q[T]$, let
        \[
            \mathcal{A}_D := \# \{F \in \mathcal{A}: D|F \}.
        \]
        If $\deg D \leq n$, then
        \begin{equation}\label{eq:size A_D}
            \mathcal{A}_D = \frac{q^n}{|D|} = \frac{X}{|D|}.
        \end{equation}
        \item We are interested in estimating the size of the sifted set
        \[
            \mathcal{S}(\mathcal{A}, \mathcal{P}, z) := \# \{F \in \mathcal{A}: (\mathbf{P}(z), F)  = 1 \} = \Psi(n,z).
        \]
    \end{enumerate}

    The assertion follows from the standard considerations of Selberg's sieve -- we provide the details for completeness.
    
    We have
    \[
        \Psi(n,z) =
        \sum_{F\in \mathcal{A}} \delta((\mathbf{P}(z),F)), 
    \]
    where $\delta(H) = \begin{cases}
            1, & \text{for } H=1\\
            0, & \text{else}
        \end{cases}$.

    Let $\Lambda=\{\lambda_D \}_{D \in \mathcal{M}}$ be a family of real parameters with $\lambda_1=1$. Then, for every $H \in \FF_q[T]$,
    \[
        \delta(H) \leq \ps{\sum_{D|H} \lambda_D}^2.
    \]
    (Our sums are always over monic polynomials.)
    Hence, 
    \[
        \Psi(n,z) \leq \sum_{F\in \mathcal{A}} \ps{\sum_{D|(\mathbf{P}(z),F)} \lambda_D}^2 = \sum_{F\in \mathcal{A}} \ps{\sum_{\substack{D|F \\ D|\mathbf{P}(z)}} \lambda_D}^2.
    \]
    Assume further that if $\deg D > z$, then $\lambda_D=0$. So, we can rewrite the right-hand side as
    \[
        \sum_{\substack{\deg D_1, \deg D_2 \leq z \\ D_1, D_2 | \mathbf{P}(z)}} \lambda_{D_1} \lambda_{D_2} \# \{F \in \mathcal{A}: D_1|F, D_2|F \} = 
        \sum_{\substack{\deg D_1, \deg D_2 \leq z \\ D_1, D_2 | \mathbf{P}(z)}} \lambda_{D_1} \lambda_{D_2} \# \{F \in \mathcal{A}: [D_1,D_2] | F \},
    \]
    Since $z \leq n/2$, we get that  $\deg [D_1,D_2] \leq n$, and hence, by \eqref{eq:size A_D}, the expression becomes
    \[
         \sum_{\substack{\deg D_1, \deg D_2 \leq z \\ D_1, D_2 | \mathbf{P}(z)}} \lambda_{D_1} \lambda_{D_2} \cdot \mathcal{A}_{[D_1,D_2]} = q^n \sum_{\substack{\deg D_1, \deg D_2 \leq z \\ D_1, D_2 | \mathbf{P}(z)}} \frac{\lambda_{D_1} \lambda_{D_2}}{|[D_1,D_2]|}.
    \]
    We write 
    \begin{equation}
        Q(\Lambda) = \sum_{\substack{\deg D_1, \deg D_2 \leq z \\ D_1, D_2 | \mathbf{P}(z)}} \frac{\lambda_{D_1} \lambda_{D_2}}{|[D_1,D_2]|},
    \end{equation}
    which is a quadratic form in the variables $\lambda_D$, $\deg D\leq z$. We will choose $\Lambda$ to minimize $Q$ under the constraint $\lambda_1=1$. We follow Selberg's approach:
    
    First, we diagonalize the form. 
    Using the equality $|F|=\sum_{D|F} \Phi(D)$, we obtain
    \begin{align*}
        Q(\Lambda) &= \sum_{\substack{\deg D_1, \deg D_2 \leq z \\ D_1, D_2 | \mathbf{P}(z)}} \frac{\lambda_{D_1} \lambda_{D_2}}{|D_1 D_2|} |(D_1,D_2)| = 
        \sum_{\substack{\deg D_1, \deg D_2 \leq z \\ D_1, D_2 | \mathbf{P}(z)}} \frac{\lambda_{D_1} \lambda_{D_2}}{|D_1 D_2|} \sum_{E|(D_1,D_2)} \Phi(E)  \\
        &= \sum_{\substack{\deg E \leq z \\ E|\mathbf{P}(z)}} \Phi(E) \sum_{\substack{E|D_1, E|D_2 \\ \deg D_1 \leq z, \deg D_2 \leq z \\ D_1|\mathbf{P}(z), D_2|\mathbf{P}(z)}} \frac{\lambda_{D_1} \lambda_{D_2}}{|D_1 D_2|} = 
        \sum_{\substack{\deg E \leq z \\ E|\mathbf{P}(z)}} \Phi(E) \ps{\sum_{\substack{E|D \\ \deg D \leq z \\ D | \mathbf{P}(z)}} \frac{\lambda_D}{|D|}}^2.
    \end{align*}
    Making the change of variables 
    \begin{equation} \label{eq:theta formula}
        \theta_E = \sum_{\substack{E|D \\ \deg D \leq z \\ D | \mathbf{P}(z)}} \frac{\lambda_D}{|D|}, \quad \deg E\leq z, \ E\mid P(z),
    \end{equation}
    we get that 
    \begin{equation}
         Q(\Lambda) = \sum_{\substack{\deg E \leq z \\ E|\mathbf{P}(z)}} \Phi(E) \theta_E^2.
    \end{equation}
    By the dual M\"{o}bius inversion formula\footnote{The dual M\"{o}bius inversion formula says that if $f(E) = \sum_{E\mid D\mid \mathbf{P}(z)} g(D)$, then $g(E) = \sum_{E\mid D\mid \mathbf{P}(z)} \mu(D/E) g(D)$. We omit the proof.} applied to \eqref{eq:theta formula}, we have
    \[
        \frac{\lambda_E}{|E|} = \sum_{\substack{E|D \\ \deg D \leq z \\ D | \mathbf{P}(z)}} \mu\bigg(\frac{D}{E}\bigg) \theta_D,
    \]
    and hence the constraint $\lambda_1=1$ transforms to
    \begin{equation}\label{constraint theta}
        L:=\sum_{\substack{\deg D \leq z \\ D| \mathbf{P}(z)}} \mu(D) \theta_D = 1.
    \end{equation}
    We apply the Lagrange multipliers method. Let $\nabla = (\partial_{\theta_E})_{\deg E\leq z, E\mid \mathbf{P}(z)}$ be the gradient operator. We want to solve the equations
    \[
        \nabla Q = C \nabla L \quad \mbox{and} \quad L=1
    \]
    in the variables $C,\theta_E$. So for each $E|\mathbf{P}(z)$ and $\deg E \leq z$, we have
    \[
        2\Phi(E) \theta_E = C \mu(E) \quad \longrightarrow \qquad \theta_E = \frac{C}{2} \frac{\mu(E)}{\Phi(E)}.
    \] 
    Plug this in \eqref{constraint theta} to get
    \[
        \frac{C}{2} = \frac{1}{\sum_{\deg D \leq z} \frac{\mu^2(D)}{\Phi(D)}}.
    \]
    Therefore, 
    \[
        \min Q(\Lambda) = \sum_{\substack{\deg E \leq z \\ E|\mathbf{P}(z)}} \frac{\mu^2(E)}{\Phi(E)\ps{\sum_{\deg D \leq z} \frac{\mu^2(D)}{\Phi(D)}}^2} = \frac{1}{\sum_{\deg D \leq z} \frac{\mu^2(D)}{\Phi(D)}} =: \frac{1}{S(z)}.
    \]
    Now, calculate 
    \begin{align*}
        S(z) = \sum_{\deg D \leq z} \frac{\mu^2(D)}{\Phi(D)} \geq \sum_{\deg D \leq z} \frac{\mu^2(D)}{|D|} &= \sum_{i=0}^z \frac{\#\{ D \in \mathcal{M}_i : D \textnormal{ is squarefree} \}}{q^i}  \\
        &= 2 + (1-\frac{1}{q})(z-2) \geq z(1-\frac{1}{q}).
    \end{align*}
    The last equality follows from \cite[Proposition~2.3]{NumberTheory_FF}.
    We finally conclude that
    \[
        \Psi(n,z) \leq q^n \min Q(\Lambda) \leq \frac{q^n}{z(1-\frac{1}{q})}, 
    \]
    as needed.
\end{proof}

The lower bound may be obtained by applying the following recursion formula for $\Psi(n,b)$. 

\begin{lemma} \label{lm:rough polynomials lower bound}
    Let $1\leq b=b(n) < n$ and $q$ a prime power. Then,
    \begin{equation}\label{eq:rec}
        n \cdot \Psi(n,b) \geq q^n - 2q^{n/2} + \sum_{b < \deg P \leq n-b-1} \deg P \cdot \Psi(n-\deg P, b).
    \end{equation}    
    Furthermore, we have
        \begin{equation}\label{eq:allb}
        \Psi(n,b) \geq \frac{q^n}{10b+5}
    \end{equation}
    and if $b$ is sufficiently large, we have
    \begin{equation}\label{eq:lb_rough_b_large}
        \Psi(n,b) \geq  \frac{q^n}{2b+2}.
    \end{equation}  
\end{lemma}

\begin{proof}
    Let us begin with the first assertion. By changing the order of summation, we get
    \begin{equation} \label{eq:recurrence}
    \begin{split}
        n \cdot \Psi(n,b) 
            &= \sum_{\substack{F \in \mathcal{M}_n \\ P^-(F) > b}} n 
            = \sum_{\substack{F \in \mathcal{M}_n \\ P^-(F) > b}} \sum_{P^k|F} \deg P  \geq \sum_{b<\deg P \leq n}  \deg P\sum_{\substack{F \in \mathcal{M}_n \\ P^-(F)>b \\ P|F}}1
            \\
            &
            = \sum_{b<\deg P \leq n}  \deg P \cdot \#\{F \in \mathcal{M}_n,P^-(F) >b,P|F\}
    \end{split}
    \end{equation}
    For a fixed $P$ with $\deg P > b$, the map $F\mapsto F/P$ gives a bijection between $\{F \in \mathcal{M}_n,P^-(F) >b,P|F\}$ and $b$-rough polynomials of degree $n-\deg P$. Moreover, if $\deg P \neq n$, then $\deg P \leq n-b-1$ since $F$ is $b$-rough. So, from \eqref{eq:recurrence} and the lower bound in \eqref{PPT}, we get that
    \[
    \begin{split}
        n \cdot \Psi(n,b) &\geq \sum_{b<\deg P \leq n} \deg P \cdot \Psi(n-\deg P, b) \\ &\geq q^n - 2q^{n/2} + \sum_{b<\deg P \leq n-b-1} \deg P \cdot \Psi(n-\deg P, b).
    \end{split}
    \]
    This finishes the proof of the first assertion.
    
    To prove \eqref{eq:lb_rough_b_large}, fix a sufficiently large $b \in \NN$. We prove by induction on $n>b$ that 
    \begin{equation}\label{eq:induction_hypothesis}
        \Psi(n,b) \geq \frac{q^n}{2b+2}.
    \end{equation} 
    If $n = b+1$, then 
    \[
        \Psi(n,b) = \pi_q(n) \geq \frac{q^n}{n} - \frac{2q^{n/2}}{n} \geq \frac{q^n}{2n} = \frac{q^n}{2b+2}.
    \]
    Now we assume \eqref{eq:induction_hypothesis} for $b < m < n$ and prove it for $n$. Using the recurrence formula \eqref{eq:rec}, we deduce
    \begin{align*}
        n \cdot \Psi(n,b) &\geq q^n - 2q^{n/2} + \sum_{b<\deg P \leq n-b-1} \deg P \cdot \Psi(n-\deg P, b) \\
        &= q^n - 2q^{n/2} + \sum_{i=b+1}^{n-b-1} \pi_q(i) \cdot i \cdot \Psi(n-i, b)
    \end{align*}
    By \eqref{PPT} and the induction hypothesis \eqref{eq:induction_hypothesis},  
    \begin{align*}
        n \cdot \Psi(n,b)&\geq q^n - 2q^{n/2} +  \sum_{i=b+1}^{n-b-1} (q^i - 2q^{i/2}) \cdot \ps{\frac{q^{n-i}}{2b+2}} = \frac{q^n(n+1)}{2b+2} - 2q^{n/2} -\frac{2q^n}{2b+2}\sum_{i=b+1}^{n-b-1} {q^{-i/2}}
    \end{align*}
    Plug in the inequality $\sum_{i=b+1}^{n-b-1}q^{-i/2} \leq q^{-(b+1)/2}(1-q^{-1/2})^{-1}$ to get 
    \begin{align*}
    n \cdot \Psi(n,b)&\geq
        \frac{q^n}{2b+2} \cdot n + q^n \ps{\frac{1-\frac{2q^{-\frac{b+1}{2}}}{1-q^{-1/2}} - 2(2b+2) \cdot q^{-n/2}}{2b+2}}.
    \end{align*}
    Since $b$ is sufficiently large, 
    \[
        1-\frac{2q^{-\frac{b+1}{2}}}{1-q^{-1/2}} - 2(2b+1) \cdot q^{-n/2} \geq 0
    \]
    hence\eqref{eq:lb_rough_b_large} follows.
    
    Finally, we prove \eqref{eq:allb}. If $b$ is sufficiently large, we are done by \eqref{eq:lb_rough_b_large}. To this end, assume that $1\leq b \leq b_0$ for some fixed $b_0$. 
    If $b=1$ and $n=2$, we have
    \begin{equation}\label{uniform lower bound rough: base 1}
        \Psi(2,1) = \pi_q(2) = \frac{q^2}{2} - \frac{q}{2} \geq \frac{q^2}{4}\geq \frac{q^n}{10b+5}.
    \end{equation} 
    
    Next,  we assume that $3 \leq n \leq 2b+1$. Then $q^{n/2} \leq \frac{q^n}{2^{3/2}}$ and $\Psi(n,b)=\pi_q(n)$. So  by \eqref{PPT} we conclude that 
    \begin{equation}\label{uniform lower bound rough: base 2}
        \Psi(n,b)\geq \frac{q^n}{n} - \frac{2q^{n/2}}{n} \geq (1-2^{-1/2})\frac{q^n}{n} \geq (1-2^{-1/2})\frac{q^n}{2b+1} \geq \frac{q^n}{10 b+5}.
    \end{equation}
    
    Finally, we assume that $n\geq 2b+2$ and proceed by induction. So, we have $\psi(m,b)\geq \frac{q^m}{10b+5}$ for all $b\leq m<n$. We repeat the argument we used to prove \eqref{eq:lb_rough_b_large} with $10b+5$ replacing $2b+2$. It gives
    \[
        n \cdot \Psi(n,b) \geq \frac{q^n}{10 b+5} \cdot n + \frac{q^n}{10 b+5} \ps{4-\frac{2q^{-\frac{b+1}{2}}}{1-q^{-1/2}} - 2(10b+5) \cdot q^{-n/2} + 8b}.
    \]
    It is immediate that the term in the brackets is positive whenever $n\geq 4$, so \eqref{eq:allb} follows.
\end{proof}

The next result shows that rough polynomials are equidistributed in arithmetic progressions up to a constant, as long as $\deg M \leq (\frac{1}{2}-\epsilon) b$.

\begin{theorem}\label{lm:rough_polynomialsAP} [Rough Polynomials in Arithmetic Progressions]
    Let $\epsilon > 0$ and let $q$ be a prime power. Let $1\leq b=b(n) < n$, and $A=A(n,q),M=M(n,q) \in\FF_q[T]$ be such that $(A,M)=1$ and $\deg M \leq b(1/2-\epsilon)$. Then,
    \[
    \Psi(n,b;A,M) \asymp_\epsilon \frac{q^n}{\Phi(M)} \cdot \frac{1}{b}, \qquad q^n \to \infty,
    \]    
uniform in all such parameters.
\end{theorem}

    Gorodetsky \cite[Theorem 2.1]{Gorodetsky} obtains estimates on character sums. These imply that $b$-rough polynomials equidistribute amongst the invertible residues in an arithmetic progression of large modulus. Hence, together with Theorem~\ref{prop:Uniform Rough Polynomials}, it proves Theorem~\ref{lm:rough_polynomialsAP}. 

    We provide here an alternative proof, which is more direct, and hence shorter and self-contained.

\begin{proof}
    Since $\deg M \leq b(1/2-\epsilon)$, by \eqref{PPTAP}, for all $j \geq b$, 
    \begin{equation}\label{hypothesis on S}
        \pi_q(j;A,M) \asymp_{\epsilon} \frac{q^j}{j \cdot \Phi(M)}, \qquad q^n \to \infty.
    \end{equation}
    If $\deg M = 0$, the assertion follows from Theorem~\ref{prop:Uniform Rough Polynomials}. So, assume $\deg M \geq 1$.
    For $F \in \mathcal{M}_n$, denote by $m_i$ the number of primes $P|F$ with $\deg P = i$, and let $\lambda(F) = (m_1,m_2,\dots,m_n)$. We denote by $(\boldsymbol{e}_1, \ldots, \boldsymbol{e}_n)$ the standard basis of $\mathbb{R}^n$. 
    For $E \in \FF_q[T]/M\FF_q[T]$, let 
    \begin{align*}
        \xi(n,\boldsymbol{m}) &:= | \{F\in \mathcal{M}_n : \lambda(F) = \boldsymbol{m}\}|\cdot 1_{\{m_1=\cdots=m_{b}=0\}},\\
        \xi(n,\boldsymbol{m},E,M) &:= |\{F\in \mathcal{M}_n : \lambda(F) = \boldsymbol{m},\  F \equiv E \mod{M}\}|\cdot 1_{\{m_1=\cdots=m_{b}=0\}}.
    \end{align*}
    Moreover, we denote
    \begin{align*}
        \vartheta &:=|\{F \in \mathcal{M}_n: P^-(F)>b, \forall j: m_j \leq 3 \}|= \sum_{\substack{ \forall j: m_j\leq 3}} \xi(n,\boldsymbol{m})
        , \\
        \vartheta_E &:= |\{F \in \mathcal{M}_n: P^-(F)>b, \forall j: m_j \leq 3, F \equiv E \mod M \}|=\sum_{\forall j: m_j \leq 3} \xi(n,\boldsymbol{m},E,M) 
        ,
    \end{align*}
    where the sums run over tuples $\boldsymbol{m} = (m_1,...,m_n)$ such that $m_i\geq 0$ and $\sum im_i = n$, which we henceforth call \emph{partitions} of $n$. 
    For a partition $\boldsymbol{m} = (m_1,m_2,...)$ of $n\geq 1$, we let  $\alpha = \alpha(\boldsymbol{m})$ be the minimal index such that $m_\alpha \neq 0$.
    
    Since $\Psi(n,b) = \sum \xi(n,\boldsymbol{m})$, 
    \begin{align}\label{eq:factor_Psi_m_jleq3}
        \Psi(n,b)   &= \vartheta+ \sum_{\exists j: m_j \geq 4} \xi(n,\boldsymbol{m}).
    \end{align}
    First, we show that the sum to the right is negligible:
    \begin{align*}
    \begin{split}
         \sum_{\exists j: m_j \geq 4} \xi(n,\boldsymbol{m}) &\leq
        \sum_{j=b+1}^n | \{F\in\mathcal{M}_n: P^-(F) > b, m_{j} \geq 4\}| \leq \sum_{j=b+1}^n \Psi(n-4j,b) \cdot \pi_q(j)^4\ .
    \end{split}
    \end{align*}
    If $\frac{k}{2}<b<k$, then $\Psi(k,b) = \pi_q(k) \leq \frac{q^k}{k} \leq \frac{2q^k}{b}$. If $b\leq k/2$, by Lemma~\ref{lm:Sieve upper bound}, we have $\Psi(k,b)\leq \frac{q^k}{b(1-q^-1)} \leq \frac{2q^k}{b}$. If $b\geq k$, then $\Psi(k,b)\leq 1$ (with equality iff $k=0$). By \eqref{PPT}, we have $\Psi(0,b)\cdot \pi_q(j)^4 \leq \frac{q^n}{j^4} \leq \frac{q^n}{(b+1)^4}$. Together with \eqref{PPT}, we obtain
    \begin{equation}\label{eq:small error}
    \begin{split}
        \sum_{\exists j: m_j \geq 4} \xi(n,\boldsymbol{m})  &\leq \sum_{j=b+1}^n \frac{2q^{n-4j}}{b} \frac{q^{4j}}{j^4} + \frac{q^n}{(b+1)^4} = \frac{2q^n}{b} \sum_{j=b+1}^n \frac{1}{j^4} + \frac{q^n}{(b+1)^4}  \\ 
        &\leq \frac{2q^n}{b} \int_b^\infty \frac{dx}{x^4} + \frac{q^n}{(b+1)^4} \leq \frac{2q^n}{3b^4} + \frac{q^n}{(b+1)^4} \leq \frac{q^n}{27b},
    \end{split}
    \end{equation}
    where in the last inequality we used the fact that $b \geq 3$, which follows from the assumption $\deg M \leq (1/2-\epsilon)b$. Inserting \eqref{eq:small error} and the bound in Lemma \ref{lm:rough polynomials lower bound} to \eqref{eq:factor_Psi_m_jleq3}, we get
    \[
        \vartheta \geq \frac{q^n}{10b+5} - \frac{q^n}{27b} \gg \frac{q^n}{b},
    \]
    uniformly in $q,n,b$.
    On the other hand, by Lemma \ref{lm:Sieve upper bound} we get that $\vartheta \leq \Psi(n,b) \leq \frac{q^n}{b(1-q^{-1})}$. Hence,
    \begin{equation}\label{eq:theta_asymp}
        \vartheta \asymp \frac{q^n}{b}.
    \end{equation}
    We show that $\vartheta$ is equidistributed (up to a constant) in arithmetic progressions.
    Let 
    $        \chi(\boldsymbol{m},i,j) := \begin{cases}
        1 & \alpha(\boldsymbol{m})=j, m_j=i \\
        0 & \textnormal{otherwise}
    \end{cases}.$
    For $D \in (\FF_q[T]/M\FF_q[T])^*$, we calculate $\vartheta_D$ by splitting the event according to the value of $j:= \alpha(\boldsymbol{m})$ and $i:=m_j$:
    \begin{equation}
    \begin{split}
        \vartheta_D &=\sum_{\forall j: m_j \leq 3} \xi(n,\boldsymbol{m},D,M)\\
        &= \sum_{i=1}^3 \sum_{j=b+1}^n \sum_{\boldsymbol{m}} \chi(\boldsymbol{m},i,j) \cdot \xi(n,\boldsymbol{m},D,M).
     \end{split}
    \end{equation}
    Now,  if $m_j=i$, we have
    \[
        \xi(n,\boldsymbol{m},D,M) = \sum_C \xi(n-ij, \boldsymbol{m}-i\boldsymbol{e}_j, C,M) \cdot \xi(ij,i\boldsymbol{e}_j,C^{-1}D,M),
    \]
    where we denote by $\sum_C$ to be the sum over monic $C \in (\FF_q[T]/M\FF_q[T])^*$.
    Using this observation, we obtain
    \begin{equation}\label{eq:theta_calc}
        \vartheta_D = \sum_C \sum_{i=1}^3 \sum_{j=b+1}^n \mathop{\sum}_{\boldsymbol{m}} \chi(\boldsymbol{m},i,j) \xi(n-ij ,\boldsymbol{m} -i \boldsymbol{e}_j,C,M)  \xi(ij,i \boldsymbol{e}_j, C^{-1}D,M).
    \end{equation}
    We denote the number of prime $l$-roots of $E$ modulo $M$ by
    \begin{align*}
        \Gamma_q(n,l,E,M) := \# \{P \in \mathcal{M}_n: P \mbox{ is prime}, P^l \equiv E \mod M \}.
    \end{align*}
    Since $\deg M \leq \frac{1}{2}j$, it follows that  $\Gamma_q(j,i,B,M) \leq \pi_q(j) \leq \frac{q^j}{j} \ll \frac{q^{2j}}{\Phi(M) j^2}$.
    Then, by \eqref{hypothesis on S}, we estimate $\xi(ij,i\boldsymbol{e}_j,B,M)$:
    \begin{equation*}
    \begin{split}
        \xi(j,\boldsymbol{e}_j,B,M) &= \pi_q(j,B,M) \asymp_{\epsilon} \frac{q^j}{\Phi(M) \cdot j}, \\
        \xi(2j,2\boldsymbol{e}_j,B,M) &= \frac{1}{2} \sum_{E} \pi_q(j,E,M) \cdot \pi_q(j,BE^{-1},M) + \frac{1}{2} \Gamma_q(j,2,B,M) \\
        &\asymp_{\epsilon} \frac{q^{2j}}{\Phi(M)j^2},
    \end{split}
    \end{equation*}

    Similarly,
    \begin{align*}
        \xi(3j,3\boldsymbol{e}_j,B,M)  &= \frac{1}{6}\sum_{E_1,E_2} \pi_q(j,E_1,M) \cdot \pi_q(j,E_2,M) \cdot \pi_q(j,(E_1E_2)^{-1}B,M) \\
        &\quad + \frac{1}{3} \sum_{E} \pi_q(j,E,M) \cdot \pi_q(j,BE^{-2}) + \frac{1}{2} \Gamma_q(j,3,B,M) \\
        &\asymp_{\epsilon} \frac{q^{3j}}{\Phi(M) j^3}.
    \end{align*}
    In particular, these estimates are independent of $B$.
    By \eqref{eq:theta_calc}, $\vartheta_D$ is also independent of $D$, up to constants, say $\vartheta_D \asymp_\epsilon \vartheta_A$. Thus,
    \begin{equation*}
        \vartheta = \sum_D \vartheta_D \asymp_\epsilon \Phi(M) \cdot \vartheta_A,
    \end{equation*}
    so by \eqref{eq:theta_asymp} we get
    \begin{equation*}
        \vartheta_A \asymp_{\epsilon} \frac{q^n}{b \cdot \Phi(M)}  .
    \end{equation*}
    Going back to $\Psi(n,b,A,M)$, we write
    \begin{equation}\label{eq:split_mod}
        \Psi(n,b,A,M)  
        = \vartheta_A + \sum_{\exists j: m_j \geq 4} \xi(n,\boldsymbol{m},A,M).
    \end{equation}
    Similarly to \eqref{eq:small error}, we bound the second term:
    \begin{align*}
        \sum_{\exists j: m_j \geq 4} \xi(n,\boldsymbol{m},A,M) &\leq \sum_{j=b+1}^n \# \{F \in \mathcal{M}_n: P^-(F) > b, m_j \geq 4, F \equiv A \mod M \} \\
        &\leq \sum_{j=b+1}^n \sum_{C} \Psi(n-4j,b,C,M) \cdot \xi(4j, 4\boldsymbol{e}_j,AC^{-1},M).
    \end{align*}
    Now, by \eqref{hypothesis on S} we have
    \begin{align*}
        \xi(4j, 4\boldsymbol{e}_j,AC^{-1},M) &\leq \sum_{E_1,E_2,E_3} \ps{\prod_{i=1}^3 \pi_q(j,E_i,M)} \cdot \pi_q(j, A(CE_1E_2E_3)^{-1},M) \\
        &\ll_{\epsilon} \frac{q^{4j}}{\Phi(M) \cdot j^4},
    \end{align*}
    so by Theorem \ref{prop:Uniform Rough Polynomials} we obtain
    \begin{align*}
        \sum_{\exists j: m_j \geq 4} \xi(n,\boldsymbol{m},A,M) &\ll_{\epsilon} \sum_{j=b+1}^n \sum_C \Psi(n-4j,b,C,M) \frac{q^{4j}}{\Phi(M) \cdot j^4} \\
        &= \sum_{j=b+1}^n \Psi(n-4j,b) \frac{q^{4j}}{\Phi(M) \cdot j^4} \\
        &\ll \sum_{j=b+1}^n \frac{q^n}{b\cdot \Phi(M) \cdot j^4} 
        \ll \frac{q^n}{\Phi(M) \cdot b^4}.
    \end{align*}
    Hence, 
    \(
        \Psi(n,b,A,M) = \vartheta_A+O_{\epsilon}\ps{\frac{q^n}{\Phi(M)\cdot b^4}} \asymp_{\epsilon} \frac{q^n}{\Phi(M) \cdot b},
    \)
    as  claimed.
\end{proof}

\section{Arithmetic Lower Bound} \label{sec:arithmetic_lb}

In this section, we bound $H(n,b;A,M)$ and $H(n,b;A_1,M_1,A_2,M_2)$ from below by a weighted sum over the number of degrees of divisors, which we bound in the next section. Our proof follows the arguments of Ford \cite{ford_H_y_2y} and their adaptation to function fields by Meisner \cite{Meisner_MT}. 
For a polynomial $H\in\mathbb{F}_q[T]$, let
\begin{align*}
    \mathcal{L}(H) &= \left\{d:d=\deg(D) \textnormal{ for some } D|H\right\}, \\
    L(H) &= |\mathcal{L}(H)|.
\end{align*}

\begin{lemma}\label{lm:main_lower_bound}
    Let $0<\epsilon<\frac{1}{27}$, $1 \ll_{\epsilon} b=b(n,q)\leq n/2$. Let $M_i, A_i\in \FF_q[T],\ i=1,2$, and $M = [M_1,M_2]$.  Assume $(A_i,M)=1$, $\deg M_i \leq (1/2-\epsilon) b$, for $i=1,2$. Then,
    \begin{align*}
        |H(n,b;A_1,A_2,M_1,M_2)| &\gg_{\epsilon} \frac{q^{n}}{\Phi(M_1)\cdot\Phi(M_2)} \cdot \frac{1}{b^2} \sum_{\substack{\deg H \leq \frac{\epsilon}{7} b \\ (H,M)=1}}  \frac{L(H)}{|H|}, \textnormal{ and }\\
        |H(n,b;A_2,M_2)| &\gg_{\epsilon} \frac{q^{n}}{\Phi(M_2)} \cdot \frac{1}{b^2} \sum_{\substack{\deg H\leq \frac{\epsilon}{7} b \\ (H,M_2)=1}}  \frac{L(H)}{|H|}, \qquad \mbox{as $q^n\to \infty$}
    \end{align*}
\end{lemma}

\begin{proof}
    We begin with the proof of the first assertion. Let $\mathcal{J} := [\frac{2\epsilon}{7} b,(1-\frac{2\epsilon}{7}) b]$.
    Consider the set of polynomials of degree $n$ of the form $F=HPB$, such that
    \begin{enumerate}[label=\arabic*.]
        \item $\deg H \leq \frac{\epsilon}{7} b$ and  $(H,M)=1$. \label{cond H}
        \item There exists $G_1 | H$ such that $\deg P = b - \deg G_1$, and $P G_1 \equiv A_1 \mod M_1$. \label{cond P} Put $E:=A_2G_1H^{-1}\mod M_2$.
        \item $B \equiv A_2 G_1H^{-1} 
        \mod M_2$ and every prime divisor $Q$ of $B$ satisfies $\deg Q \in \mathcal{J}$ or $\deg Q > b$. \label{cond B}
    \end{enumerate}
    Condition~\ref{cond P}\ implies that $F\in H(n,b;A_1,A_2,M_1,M_2)$ (with the divisors $G_1P$ and $HB/G_1$). It also implies that $(1-\frac{\epsilon}{7}) b \leq \deg (P)\leq b$. Therefore, this representation is unique and it suffices to bound from below the number of such triples $(H,P,B)$. 
    
    We estimate the number of $B$ for given $H$ and $P$ that satisfy \ref{cond H}\ and \ref{cond P} 
    Since $\deg P \leq b$ and $\deg H \leq \frac{\epsilon}{7}b$ we have $\deg HP\leq (1+\frac{\epsilon}{7})b$, and therefore $\deg B\geq (1-\frac{\epsilon}{7})b$.
    
    The number of $B$ with $\deg B>b$ is bounded below by the number of $b$-rough polynomials of degree $n-\deg HP$ in the arithmetic progression $A_2G_1H^{-1}\mod M_2$.  
    Since $(HP,M_2)=1$ and $\deg M_2 \leq (\frac{1}{2}-\epsilon)b \leq (\frac{1}{2}-\epsilon)\deg B$, Theorem~\ref{lm:rough_polynomialsAP} yields the lower bound
    \[
        \Psi(n-\deg HP,b; A_2G_1H^{-1}, M_2)
        \gg_{\epsilon} \frac{q^n}{\Phi(M_2)}\cdot \frac{1}{b|HP|}.
    \] 
    
    Next, assume that $\deg B \leq b$. Thus, $B$ does not have prime factors of degree $>b$, and by \eqref{cond B}, it has at least $2$ prime divisors in $\mathcal{J}$. Let $\mathcal{I} := [\deg B - b + \frac{3\epsilon}{7}b, \deg B - b + \frac{5 \epsilon}{7}b]$. Since $(1-\frac{\epsilon}{7})b \leq \deg B \leq b$, we have $\mathcal{I} \bigsubset \mathcal{J}, \deg B - \mathcal{I} \bigsubset \mathcal{J}$.
    Therefore, the number of such $B's$ is bounded from below by
    \begin{equation*}
        \Xi := \sum_{d_1 \in \mathcal{I}}
        \sum_{\substack{P_1 \textnormal{ is prime} \\ \deg P_1 = d_1}} \pi_q(\deg B - d_1;(P_1)^{-1} E, M_2).
    \end{equation*}
    Notice that $\deg M_2 \leq (\frac{1}{2} - \epsilon)b$, and for $d_2 \in \deg B - \mathcal{I}$ we have $d_2 \geq (1-\frac{5\epsilon}{7})b$. Thus, by \eqref{PPT} and \eqref{PPTAP} we obtain
    \begin{align}
    \begin{split}
        \Xi &\gg_{\epsilon} \sum_{d_1 \in \mathcal{I}}
        \sum_{\substack{P_1 \textnormal{is prime} \\ \deg P_1 = d_1}} \frac{q^{\deg B - d_1}}{\Phi(M_2) \cdot (\deg B - d_1)}  \\ 
        &\gg \sum_{d_1 \in \mathcal{I}} \frac{q^{\deg B}}{\Phi(M_2)} \cdot \frac{1}{d_1(\deg B - d_1)}  
        \gg_{\epsilon} \frac{q^n}{\Phi(M_2)} \cdot \frac{1}{b|HP|}.
    \end{split}
    \end{align}

    In summary, for any choice of such $P$ and $H$, the number of corresponding $B$'s is $\gg_{\epsilon} \frac{q^n}{\Phi(M_2)} \cdot \frac{1}{b|HP|}$.
    Therefore,
    \begin{align*}
        |H(n,b; A_1,A_2,M_1,M_2)| &\gg_{\epsilon} \sum_{\substack{\deg H \leq \frac{\epsilon}{7}b \\ (H,M)=1}} \sum_{\substack{b-\deg P \in \mathcal{L}(H) \\ P \equiv A_1G_1^{-1} \mod M_1}} \frac{q^n}{b \cdot \Phi(M_2)|HP|}  \\
        &\gg_{\epsilon} \frac{q^n}{b \cdot\Phi(M_2)} \sum_{\substack{\deg H \leq \frac{\epsilon}{7}b \\ (H,M)=1}} \frac{1}{|H|} \sum_{\substack{b-\deg P\in \mathcal{L}(H) \\ P \equiv A_1G_1^{-1} \mod M_1}} \frac{1}{|P|}.
    \end{align*}
    We bound the inner sum. Fix $H$ that satisfy \ref{cond H}
    Then,
    \[
        \sum_{\substack{P \\ b- \deg P \in\mathcal{L}(H) \\ P \equiv A_1D^{-1} \mod M_1}} \frac{1}{|P|} = \sum_{\substack{d \\ b-d\in\mathcal{L}(H)}} \frac{\pi_q(d; A_1G_1^{-1}, M_1)}{q^d} \gg_{\epsilon}
        \frac{1}{\Phi(M_1)} \sum_{\substack{d \\ b-d\in\mathcal{L}(H)}} \frac{1}{d} \gg \frac{1}{\Phi(M_1)} \cdot \frac{L(H)}{b},
    \]
    by \eqref{PPTAP}, using $\deg M_1 \leq  (\frac{1}{2} - \epsilon) b$, $d \geq (1-\frac{\epsilon}{7})b$. This finishes the proof of the first assertion.
    
    For the second assertion, repeat the above argument without restricting $PG_1\equiv A_1 \pmod{M_1}$. This gives
    \begin{align*}
        |H(n,b; A_1,A_2,M_1,M_2)| &\geq \sum_{\substack{\deg H \leq \frac{\epsilon}{7}b \\ (H,M)=1}} \sum_{b-\deg P \in \mathcal{L}(H)} |\{ \mbox{eligible $B$}\}|  \\
        &\gg_{\epsilon} \frac{q^n}{\Phi(M_2)} \cdot \frac{1}{b} \sum_{\substack{\deg H \leq \frac{\epsilon}{7}b \\ (H,M)=1}} \frac{1}{|H|} \sum_{b-\deg P\in \mathcal{L}(H) } \frac{1}{|P|}.
    \end{align*}
    In this case,
    \[
        \sum_{\substack{P \\ b- \deg P\in\mathcal{L}(H)}} \frac{1}{|P|} = \sum_{\substack{d \\ b-d\in\mathcal{L}(H)}} \frac{\pi_q(d)}{q^d} \gg_{\epsilon}
         \sum_{\substack{d \\ b-d\in\mathcal{L}(H)}} \frac{1}{d} \gg\frac{L(H)}{b},
    \]
    and the proof is completed.
\end{proof}

\section{Average Number of Degrees of Divisors} \label{sec:avg_deg_div}
The goal of this section is to bound from below the sums that appear in Lemma~\ref{lm:main_lower_bound}.
\begin{proposition}\label{lower_bound}
    Let $0<\epsilon, \eta$, $0<C$, and let $1 \ll_{C,\epsilon,\eta} b=b(n)\leq n/2$. Then,
    \[
        \frac{q^{n}}{b^2} \sum_{\substack{ \deg H \leq \eta b \\ (H,M)=1}}  \frac{L(H)}{|H|} \gg_{C,\epsilon, \eta} \frac{q^n}{b^\delta(1+\log(b))^{3/2}}, \qquad q^n \to \infty,
    \]
    uniformly over all $M\in \FF_q[T]$ such that $\deg M \leq b$ and $P_i(M)\leq C \cdot \frac{q^{i}}{i^{1+\epsilon}}$ for all $i\in\NN$.
\end{proposition}

Let $H \in \mathcal{M}$. We denote by $\tau(H)$ the number of monic divisors of $H$, and by $\tau_d(H)$ the number of monic divisors of $H$ of degree $d$, so that $\tau(H) = \sum_{d\geq 0} \tau_d(H)$. Moreover, we let
\begin{equation}\label{eq:W(H) def}
    W(H) := \sum_{d\in \mathcal{L}(H)} \tau_d^2(H) = |\left\{(D,D'):D,D'|H, \deg (D) = \deg (D') \right\} |
\end{equation}
We consider only monic divisors in all definitions above.
We will be interested in the number of primes of a given degree which are relatively prime to $M$, 
\[ 
    \pi'_q(i) := |\{ P \in \mathcal{M}_i : P \mbox{ is prime and } P \nmid M \}| = \pi_q(i)-P_i(M).
\]
The assumption $P_i(M)\leq Cq^{i}/i^{1+\epsilon}$ and \eqref{PPT}, implies that 
\begin{equation}\label{eq:primes_not_m}
    \pi'_q(i) = \frac{q^i}{i} + O\ps{\frac{q^i}{i^{1+\epsilon}}}, \quad  \forall i\in \NN.
\end{equation}

Define the sequence of integers $1=\lambda_1 < \lambda_2<\cdots $ inductively by the property that $\lambda_j$ is the largest integer such that 
\begin{equation}\label{eqn:Ebound}
    \sum_{\substack{\deg P\in (\lambda_{j-1},\lambda_j] \\ P \nmid M}}|P|^{-1}\leq \log 2.
\end{equation}
By \eqref{eq:primes_not_m}, $\displaystyle\sum_{\deg P>\lambda_{j-1}}|P|^{-1} \gg \sum_{i>\lambda_{j-1} }i^{-1}=\infty$, so $\lambda_j$ exists. By \eqref{PPT}, $\displaystyle\sum_{\deg P = \lambda_{j-1}+1} |P|^{-1} \leq (\lambda_{j-1}+1)^{-1}< \log 2$, and hence $\lambda_j>\lambda_{j-1}$ (in particular, $\lambda_j>j$). 
Let $E_j$ be the set of prime polynomials $P\nmid M$ such that $\deg P\in (\lambda_{j-1},\lambda_{j}]$.

We apply \eqref{eq:primes_not_m} to estimate the sum of reciprocal of elements in $E_j$:
\begin{align*}
    \sum_{P\in E_j} \frac{1}{|P|} &= \sum_{i=\lambda_{j-1}+1}^{\lambda_j} \frac{\pi'_q(i)}{q^i} 
    = \sum_{i=\lambda_{j-1}+1}^{\lambda_j} \frac{1}{i} + O\ps{\frac{1}{i^{1+\epsilon}}} \\
     &=\log(\lambda_j) - \log(\lambda_{j-1}) + O\ps{\frac{1}{\lambda_{j-1}^{\epsilon}}}.  
\end{align*}
By \eqref{eqn:Ebound}, $\log(\lambda_j) - \log(\lambda_{j-1}) + O\ps{\frac{1}{\lambda_{j-1}^{\epsilon}}}\leq \log 2$ and from the maximality of $\lambda_j$ and the fact that $\sum_{\deg P=\lambda_{j}+1} = O(\lambda_{j}^{-1})$, we get $\log(\lambda_j) - \log(\lambda_{j-1}) + O\ps{\frac{1}{\lambda_{j-1}^{\epsilon}}}\geq \log 2$. 

Therefore,  
there exists some constant $K = K(C,\epsilon)$ such that for all $j$,
\begin{equation}\label{eqn:lambda_bound}
    2^{j-K} \leq \lambda_j \leq 2^{j+K}.
\end{equation}

Finally, for a tuple of integer $v = (b_1,...,b_J)$, let $\mathcal{A}(v)$ be the set of squarefree monic polynomials with exactly $b_j$ prime divisors from the set $E_j$, $j=1,\ldots, J$ and no other prime factors.
\begin{lemma}\label{lm:W(A)}
For a given $v$ with $b_j\leq Cq^j/j^{1+\epsilon}$ for all $j=1,\ldots, J$,  we have
\[
    \sum_{H\in \mathcal{A}(v)} \frac{W(H)}{|H|} \ll_{C,\epsilon} \frac{(2\log(2))^{b_1+\dots+b_J}}{b_1!\dots b_J!} \sum_{j=1}^{J} 2^{-j+b_1+\dots+b_j},
\]
as $q^n\to \infty$.
\end{lemma}
\begin{proof}
    Let $B = b_1+ \dots + b_J$. Each $H \in \mathcal{A}(v)$ has a factorization $H = \prod_{i=1} ^ {B} P_i$ such that
    \begin{equation}\label{eqn:intervals}
        P_1, \dots , P_{b_1} \in E_1, P_{b_1 + 1}, \dots , P_{b_1+b_2} \in E_2, \ldots
    \end{equation}
    For later use, for an $1\leq i\leq B$, we write $j_0(i)$ for the index such that $P_i\in E_{j_0(i)}$.
    Then, $W(H)$ equals the number of subsets $Y,Z \subseteq \{1,...,B\}$ such that
    \begin{equation}\label{eqn:simdiv}
        \sum_{i\in Y} \deg P_i = \sum_{i \in Z} \deg P_i.
    \end{equation}
    Changing the order of summation gives the following.
    \begin{equation}\label{eqn:upperbound}
        \sum_{H \in \mathcal{A}(v)} \frac{W(H)}{|H|} \leq \frac{1}{b_1! \dotsb b_J!} \sum_{Y,Z \subseteq \{1, \dotsc ,B\}} \mathop{{\sum}'}_{(P_1,\dotsc,P_B)}\frac{1}{|P_1| \dotsb |P_B|},
    \end{equation}
    whereas the ${\sum}'$ indicates that $P_1,\ldots, P_B$  are the prime factors of $H$ in the above factorization.
    First, consider the diagonal term, where $Y=Z$,
    \begin{equation}
        \sum_{Y \subseteq \{1, \dotsc ,B\}} \mathop{{\sum}'}_{(P_1,\dotsc,P_B)}\frac{1}{|P_1| \dotsb |P_B|} \leq \sum_{Y \subseteq \{1,\dotsc ,B\}} \prod_{j=1}^{J} \left( \sum_{P_j \in E_j} \frac{1}{|P_j|} \right)^{b_j} \leq (2 \log (2))^B,
    \end{equation}
    where the last inequality is due to \eqref{eqn:Ebound}.
    
    When $Y \neq Z$, we let $I := \max (Y \Delta Z)$, where $\Delta$ denotes the symmetric difference. For all fixed $P_i$, $i \in (Y\cup Z)\smallsetminus\{ I\}$. The degree $d = d(P_i, i\neq I)$ of any suitable choice of $P_I$ is independent of the choice of $P_I$ itself. 
    Write $j_0:= j_0(I)$. By \eqref{eqn:lambda_bound} we have 
    \[
        d = \deg (P_I) \geq \lambda_{j_0-1} \gg_{C,\epsilon} 2^{j_0}.
    \]
    Thus, by \eqref{PPT},
    \begin{equation}
        \sum_{\deg P_I=d} \frac{1}{|P_I|} \leq \frac{\pi_q(d)}{q^{d}} \ll \frac{1}{d} \ll_{C,\epsilon} 2^{-j_0}.
    \end{equation}
    This implies that for fixed $Y \neq Z$ we have
    \begin{equation}\label{eq:boubndsum'}
    \begin{split}
        \mathop{{\sum}'}_{(P_1,\dotsc,P_B)}\frac{1}{|P_1| \dotsb |P_B|}  &\leq
        \sum_{i\neq I} \sum_{P_i\in E_{j_0(i)}} \prod_{i\neq I}\frac{1}{|P_i|} \sum_{\deg P_I=d(P_i:i\neq I)} \frac{1}{|P_I|}\\
        &\leq 2^{-j_0} \sum_{i\neq I} \sum_{P_i} \prod_{i\neq I}\frac{1}{|P_i|} \leq 2^{-j_0} \prod_{i\neq j_0} \left( \sum_{P_j \in E_j} \frac{1}{|P_j|} \right)^{b_j} \\
        &\ll (\log 2)^B 2^{-j_0}
    \end{split}
    \end{equation}
        
    The number of subsets $Y\neq Z$ with $I=\max(Y\Delta Z)$ is $2^{B+I-1}$. 
    So, together with \eqref{eq:boubndsum'}, we deduce that 
    \begin{align*}
        \sum_{H \in \mathcal{A}(v)} \frac{W(H)}{|H|} &\ll_{C,\epsilon} \frac{(2\log(2))^B}{b_1 ! \dotsb b_J !} \left(1 + \sum_{I=1}^B 2^{-j_0(I)} 2^{I-1}\right) \\
        &\ll \frac{(2\log(2))^B}{b_1 ! \dotsb b_J !} \sum_{j=1}^J 2^{-j} \sum_{I: j_0(I) = j} 2^I 
        \ll \frac{(2\log(2))^B}{b_1 ! \dotsb b_J !} \sum_{j=1}^J 2^{-j+b_1+\dotsb +b_j},
    \end{align*}
    where the last inequality follows from the fact that $j_0(I)=j$ if and only if $b_1+\dotsb+b_{j-1} < I \leq b_1+ \dotsb + b_j$.
    
\end{proof}

\begin{lemma}\label{lm:tau(A)}
    Suppose that there exists $N>0$ such that $b_i=0$ for $i<N$ and $b_j \leq Nj$ for $j\leq J$. Then,
    \[
        \sum_{H\in\mathcal{A}(v)} \frac{\tau (H)}{|H|} \gg_{N,C,\epsilon} \frac{(2 \log(2))^{b_N+\dots+b_J}}{b_N!\dots b_J!},
    \]
as $q^n \to \infty$.
\end{lemma}

\begin{proof}
    If $H \in \mathcal{A}(v)$, then $\tau(H) = 2^{b_N+ \dotsb +b_J}$. Hence,
    \begin{equation}\label{boundtHoverH}
        \sum_{H\in\mathcal{A}(v)} \frac{\tau (H)}{|H|} = 2^{b_N+ \dotsb +b_J} \prod_{j=N}^J \frac{1}{b_j !} \left(\sum_{\substack{P_1, \dotsc, P_{b_j} \in E_j \\ P_i \textnormal{ distinct}}} \frac{1}{|P_1| \dotsb |P_{b_j}|}\right).
    \end{equation}
    By \eqref{PPT} and the choice of the sequence of $\lambda_j$'s, for $j \geq N$,
    \begin{align*}
        \sum_{P \in E_j} \frac{1}{|P|} \geq \log(2) - \sum_{\substack{\deg(P)=\lambda_{j+1} \\ P \nmid M}} \frac{1}{|P|} &\geq \log(2) - \frac{\pi_q(\lambda_{j+1})}{q^{\lambda_{j+1}}} 
        \geq \log(2) - \frac{1}{\lambda_{j+1}}.
    \end{align*}
    For $k\leq b_j$ and fix distinct $P_1, \dotsc ,P_k \in E_j$, we deduce that
    \begin{equation*}
        \sum_{\substack{P \in E_j \\ P \neq P_1, \dotsc ,P_k}} \frac{1}{|P|} = \sum_{P \in E_j}\frac{1}{|P|} - \sum_{j=1}^k \frac{1}{|P_i|} \geq \log(2) - \frac{1}{\lambda_{j+1}} - \frac{b_j}{q^{\lambda_{j-1}}}.
    \end{equation*}
    Hence, we may bound the product in \eqref{boundtHoverH}:
    \begin{align*}
        \prod_{j=N}^J \frac{1}{b_j !} \left(\sum_{\substack{P_1, \dotsc, P_{b_j} \in E_j \\ P_i \textnormal{ distinct}}} \frac{1}{|P_1| \dotsb |P_{b_j}|}\right) &\geq \prod_{j=N}^J \frac{1}{b_j !} \left( \log(2) - \frac{1}{\lambda_j+1} - \frac{b_j}{q^{\lambda_{j-1}}}\right)^{b_j} \\
        &= \frac{\log(2)^{b_N + \dotsb + b_J}}{b_N!\dots b_J!}  \prod_{j=N}^J \left( 1 - \frac{1}{\log(2)} \left(\frac{1}{\lambda_j+1} + \frac{b_j}{q^{\lambda_{j-1}}}\right)\right)^{b_j}.
    \end{align*}
    To conclude the proof, it remains to bound the right-hand product from below.
    And indeed, as 
    \[
        \frac{1}{\log(2)} \left(\frac{1}{\lambda_j+1} + \frac{b_j}{q^{\lambda_{j-1}}}\right) \leq A 2^{-j}, 
    \]
    for some constant $A=A(C,\epsilon)$, 
    by the Weierstrass inequality $\prod_{i}(1-x_i)^{w_i}\geq 1-\sum_{i}w_ix_i$ and the assumption $b_j\leq Nj$, we get that 
    \[
        \prod_{j=N}^J \left( 1 - \frac{1}{\log(2)} \left(\frac{1}{\lambda_j+1} + \frac{b_j}{q^{\lambda_{j-1}}}\right)\right)^{b_j} \geq 1-A N\sum_{j=N}^{J} j2^{-j} \gg_{N,C,\epsilon} 1, 
    \]
    as needed.
\end{proof}

\begin{proof}[Proof of Proposition \ref{lower_bound}]
    We have
    \[
        \tau (H) = \sum_{d \in \mathcal{L}(H)} \tau_d(H),
    \]
    and by 
    \eqref{eq:W(H) def}, we have
    \[
        W(H) = \sum_{d \in \mathcal{L}(H)} \tau_d(H)^2.
    \]
    
    For a nonempty finite set $\mathcal{A}\subseteq \FF_q[T]$ of monic polynomials, we apply the Cauchy-Schwarz's inequality in the following way.
    \begin{align*}
        \Big( \sum_{H \in \mathcal{A}} \frac{\tau(H)}{|H|} \Big) ^2 &= 
        \Big( \sum_{H \in \mathcal{A}} \sum_{d \in \mathcal{L}(H)} \frac{\tau_d(H)}{|H|} \Big) ^2  
        \leq
        \Big( \sum_{H \in \mathcal{A}} \sum_{d \in \mathcal{L}(H)} \frac{1}{|H|} \Big) \Big( \sum_{H \in \mathcal{A}} \sum_{d \in \mathcal{L}(H)} \frac{\tau_d(H)^2}{|H|} \Big) \\
        &= \Big( \sum_{H \in \mathcal{A}} \frac{L(H)}{|H|} \Big) 
        \Big( \sum_{H \in \mathcal{A}} \frac{W(H)}{|H|} \Big).
    \end{align*}
    So 
    \begin{equation}\label{concfromCS}
         \sum_{H \in \mathcal{A}} \frac{L(H)}{|H|}  \geq  \frac{ (\sum_{H \in \mathcal{A}} \frac{\tau(H)}{|H|})^2 }{ \sum_{H \in \mathcal{A}} \frac{W(H)}{|H|} }.
    \end{equation}
    Let $\mathcal{A}_1,...,\mathcal{A}_l$ be a collection of disjoint nonempty finite sets of polynomials of degrees $\leq \eta b$ that are relatively prime to $M$. We apply \eqref{concfromCS} to each of them to get 
    \begin{equation}\label{eqn:cs_geq}
        \frac{q^n}{b^2} \sum_{\substack{\deg(H) \leq \eta b \\ (H,M)=1}} \frac{L(H)}{|H|} \geq \frac{q^n}{b^2} \sum_{i \in I} \frac{( \sum_{H \in \mathcal{A}_i} \frac{\tau(H)}{|H|} ) ^2}{ \sum_{H \in \mathcal{A}_i} \frac{W(H)}{|H|} }.
    \end{equation}

    To apply \eqref{eqn:cs_geq}, we choose the sets $\mathcal{A}_i$ to be of the form $\mathcal{A}(v)$, for vectors $v$ starting with many zeros. 
    
    More precisely, 
    let $N=N(C,\epsilon,\eta)$ be an integer such that $N 2^{K+1-N}  \leq \eta$ (where $K$ is the constant defined in \eqref{eqn:lambda_bound}). Set $k = \lfloor\log_2(b)-2N\rfloor$ and $J := N+k-1$. Let $\mathcal{B}$ be the set of vectors $v = (b_1, \dotsc ,b_J)$ such that $b_1=\cdots =b_{N-1}=0$, $b_N+ \dotsb + b_J=k$, and $b_j \leq N \min(j,J-j+1)$. Using \eqref{eqn:lambda_bound}, for every $H \in \mathcal{A}(v)$, we have
    \begin{align*}
        \deg(H) 
            &\leq \sum_{j=N}^J b_j \lambda_j 
            \leq N2^{K+J+1} \sum_{j=N}^J (J-j+1)2^{j-J-1} \leq   
              N2^{K+J+1}\sum_{l=1}^{\infty} l 2^{-l} \\
        &= N2^{K+J+2} = N2^{K+1-N} \times 2^{2N+k} \leq 
        \eta b.
    \end{align*}
    Therefore, \eqref{eqn:cs_geq} applied to the sets $\mathcal{A}(v)$, $v\in\mathcal{B}$ gives
    \begin{equation}\label{eqn:lowerbdintermsofAv}
        \frac{q^n}{b^2} \sum_{\substack{\deg(H) \leq \eta b \\ (H,M)=1}} \frac{L(H)}{|H|} \geq \frac{q^n}{b^2} \sum_{v \in \mathcal{B}} \frac{\left( \sum_{H \in \mathcal{A}(v)} \frac{\tau(H)}{|H|} \right) ^2}{ \sum_{H \in \mathcal{A}(v)} \frac{W(H)}{|H|} }.
    \end{equation}
    To this end, let
    \[
        f(v) = \sum_{h=N}^J 2^{N-1-h+b_N + \dotsb +b_h}.
    \]
    Lower bounds on the denominators on the right-hand side of \eqref{eqn:lowerbdintermsofAv} are given by Lemma \ref{lm:W(A)}:
    \begin{equation}
        \sum_{H\in \mathcal{A}(v)} \frac{W(H)}{|H|} \ll_{C,\epsilon} \frac{(2\log(2))^k}{b_N!\dots b_J!} (1+2^{1-N} f(v) ) \ll \frac{(2\log(2))^k}{b_N!\dots b_J!} f(v)
    \end{equation}
    where the last inequality is true since $N$ is fixed.
    Upper bounds for the numerators are given by Lemma \ref{lm:tau(A)}. Thus, we conclude that 
    \[
        \frac{q^n}{b^2} \sum_{\substack{\deg(H) \leq \eta b \\ (H,M)=1}} \frac{L(H)}{|H|} 
        \gg_{C,\epsilon,\eta} \frac{q^n (2\log(2))^k}{b^2} \sum_{v \in \mathcal{B}} \frac{1}{b_N! \dotsb b_J! f(v)}.
    \]
    In \cite[Page 8]{ford_H_y_2y}, it is proven that for a sufficiently large $k$ we have
    \[
        \sum_{v \in \mathcal{B}} \frac{1}{b_N! \dotsb b_J! f(v)} \gg \frac{k^{k-1}}{k!} \gg \frac{1}{k^{3/2}}.
    \]
    Hence, since $k = \log(b) / \log(2) + O(1)$, we obtain that for sufficiently large $b$,
    \begin{equation*}
        \frac{q^{n}}{b^2} \sum_{\substack{\deg(H)\leq \eta b \\ (H,M)=1}}  \frac{L(H)}{|H|} \gg \frac{q^n}{b^\delta(1+\log(b))^{3/2}},
    \end{equation*}
    as stated.
\end{proof}

\section{Proofs of The Main Results}\label{sec:last}
We first consider Theorems~\ref{thm: divisors on APs} and \ref{thm:MTP on AP}:
\begin{proof}[Proof of Theorems \ref{thm: divisors on APs} and \ref{thm:MTP on AP}]
    First assume that $b\gg_{C,\epsilon} 1$, that is, that there exists $b_0=b_0(C,\epsilon)$ such that $b\geq b_0$. By Lemma~\ref{lm:main_lower_bound} we have
    \[
    \begin{split}
        |H(n,b;A,M)| &\gg_{\epsilon} \frac{1}{\Phi(M)} \cdot \frac{q^{n}}{b^2} \sum_{\substack{\deg(H)\leq \frac{\epsilon}{7} b \\ (H,M)=1}} \frac{L(H)}{|H|}, \qquad \mbox{ and }\\
        |H(n,b;A_1,A_2,M_1,M_2)| &\gg_{\epsilon} \frac{1}{\Phi(M_1)\cdot \Phi(M_2)} \cdot \frac{q^{n}}{b^2} \sum_{\substack{\deg(H)\leq \frac{\epsilon}{7} b \\ (H,M)=1}} \frac{L(H)}{|H|}, \qquad \mbox{ as } q^n\to\infty.
    \end{split}
    \]
    Applying the bound of Proposition~\ref{lower_bound} (with $\eta=\epsilon/7$) to the first inequality implies Theorem~\ref{thm:MTP on AP} and to the second implies Theorem~\ref{thm: divisors on APs} for $b\gg_{C,\epsilon}1$.

    Next, assume that $b \leq b_0$.
    If $n=2b$, then, by \eqref{PPTAP} we get
    \begin{align*}
        |H(n,b;A,M)| &\geq \frac{1}{2} \sum_{(E,M)=1} \pi_q(b,E,M) \cdot \pi_q(b,AE^{-1},M) 
        \gg_{\epsilon} \frac{q^n}{\Phi(M)}.
\\    
        |H(n,b;A_1,M_1,A_2,M_2)| &\geq \frac{1}{2} \pi_q(b,A_1,M_1) \cdot \pi_q(b,A_2,M_2) \gg_{\epsilon} \frac{q^n}{\Phi(M_1) \cdot \Phi(M_2)},
    \end{align*}
    and the proof is done. (Here and below $E^{-1}$ is a polynomial representative of the inverse of $E\mod M$.)
    
    To this end, assume $n>2b$. Let $\mathcal{M}_{k,E,M} := \{F \in \mathcal{M}_k: F \equiv E \mod M\}$. If $\deg M < k$, then $|\mathcal{M}_{k,E,M}| =q^k/|M|$.
    Consider the set of polynomials of the form $F = GH$ such that $G \in \mathcal{M}_{b}$, $H \in \mathcal{M}_{n-b}$, $P^-(H) > b$, and $F \equiv A \mod M$. This presentation is unique. By Theorem \ref{lm:rough_polynomialsAP},
    \begin{align*}
        |H(n,b;A,M)| &\geq \sum_{(E,M)=1} |\mathcal{M}_{b,E,M}| \cdot \Psi(n-b, b, AE^{-1}, M) 
        \\&\gg_{\epsilon} \sum_{(E,M)=1} \frac{q^{b}}{|M|} \cdot \frac{q^{n-b}}{b \cdot \Phi(M)}= \frac{q^n}{b \cdot |M|} 
        \geq \frac{q^n}{b_0 \cdot |M|} \asymp_{C,\epsilon} \frac{q^n}{\Phi(M)}.
    \end{align*}
    The last asymptotic equality is true because $b_0$ depends only on $C$ and $\epsilon$ and  Condition~\ref{eq:consition_Pj} implies that $|M| \asymp_{C, \epsilon} \Phi(M)$.
    Similarly, 
    \begin{align*}
        |H(n,b;A_1,M_1,A_2,M_2) &\geq |\mathcal{M}_{b,A_1,M_1}|\cdot \Psi(n-b, b, A_2, M_2) \\
        &\gg_{\epsilon} \frac{q^n}{b_0 \cdot |M_1| \cdot \Phi(M_2)} \asymp_{C,\epsilon}
        \frac{q^n}{\Phi(M_1) \cdot \Phi(M_2)}.
    \end{align*}
    So the proof is complete.
\end{proof}
\begin{proof} [Proof of Theorem \ref{thm:divisor on AP}]
    The sets
    \[
        \{H(n,b;A,A',M,M)\}_{A'\in(\FF_q[T]/M\FF_q[T])^*}
    \]
    are pairwise disjoint, since polynomials from different sets lie in different arithmetic progressions. Moreover, 
    \[
        \bigcup_{A'\in(\FF_q[T]/M\FF_q[T])^*} H(n,b;A,A',M,M) \bigsubset H'(n,b;A,M).
    \]
    Therefore, by Theorem~\ref{thm: divisors on APs},
    \begin{align*}
        |H'(n,b;A,M)| &\geq \sum_{A'\in(\FF_q[T]/M\FF_q[T])^*} |H(n,b;A,A',M,M)| \\
        &\gg_{C,\epsilon} \sum_{A'\in(\FF_q[T]/M\FF_q[T])^*} \frac{1}{(\Phi(M))^2} \cdot \frac{q^n}{b^\delta (1+\log b)^{3/2}} \\
        &= \frac{1}{\Phi(M)} \cdot \frac{q^n}{b^\delta (1+\log b)^{3/2}},
    \end{align*}
    as needed.
\end{proof}
\begin{remark}
    The argument used for the last proof could not be applied to bound $H(n,b;A,M)$,
    since the corresponding sets 
    \[
        \{H(n,b;A',A(A')^{-1},M,M)\}_{A'\in(\FF_q[T]/M\FF_q[T])^*}
    \]
    are not necessarily disjoint.
\end{remark}

\begin{proof}[Proof of Theorem \ref{thm:M is bounded}]
    Since $M_1, M_2, q$ are fixed, we have $\Phi(M_1) \cdot \Phi(M_2) = O(1)$. Therefore, the lower bound for the size of each of the sets follows from the respective theorem.  All sets are contained in $H(n,b)$, so the upper bound follows from \cite[Theorem 1.2]{Meisner_MT}.
\end{proof}

\bibliographystyle{plain}

\end{document}